\documentclass[12pt,a4paper,notitlepage]{article}

\usepackage[top=4.2cm, bottom=2.8cm, left=3cm, right=3cm]{geometry}

\usepackage{mathrsfs}
\usepackage[francais, english]{babel}
\usepackage[utf8]{inputenc}
\usepackage[T1]{fontenc}
\usepackage{amsfonts, amssymb, amsthm}
\usepackage{dsfont}
\usepackage{thmtools}
\usepackage{mathtools}
\usepackage{amsmath}
\usepackage{multirow}
\usepackage{graphicx}
\usepackage{caption, subcaption}
\usepackage{float}
\usepackage{enumitem}
\usepackage{pict2e}

\usepackage{mdframed}

\usepackage[all]{xy}

\usepackage{tikz}

\usepackage{multicol}

\mathtoolsset{showonlyrefs}

\usepackage{hyperref}
\usepackage{url}

 {\begin{list}{}%
         {\setlength{\leftmargin}{#1}}%
         \item[]%
 }
 {\end{list}}
 
\allowdisplaybreaks

\newcounter{TheoremA}
\newtheorem{ThA}[TheoremA]{Theorem}

\newtheorem{Th}{Theorem}[section]
\newtheorem{Lemma}[Th]{Lemma}
\newtheorem{Prop}[Th]{Proposition}

\theoremstyle{definition}
\newtheorem{Def}[Th]{Definition}
\newtheorem{Example}[Th]{Example}

\theoremstyle{remark}
\newtheorem{Remark}[Th]{Remark}

\newcounter{fig}

\newcommand{\Lie}{\mathcal{L}}
\newcommand{\R}{\mathbb{R}}
\newcommand{\C}{\mathbb{C}}
\newcommand{\scal}[2]{\left\langle#1,#2\right\rangle}

\DeclareMathOperator{\Vol}{Vol}

\newcommand{\norm}[1]{\left\vert \left\vert#1 \right\vert \right\vert }

\usepackage{relsize,exscale}

\title{Curvature of the Completion of the Space of Sasaki Potentials.}
\author{Thomas Franzinetti}
\date{}

\begin{document}
\pagenumbering{arabic}
\begin{center}
\LARGE

\vspace*{1cm}
\textbf{Curvature of the Completion of the Space \\ of Sasaki Potentials.}

\large
\vspace{0.5cm}
Thomas Franzinetti
\end{center}

\normalsize
\vspace{1cm}
\noindent
\textbf{Abstract:} Given a compact Sasaki manifold, we endow the space of the Sasaki potentials with an analogue of Mabuchi metric. We show that its metric completion is negatively curved in the sense of Alexandrov.

\vspace{0.3cm}

\vspace{0.3cm}

\section{Introduction}

After being introduced in 1960 by S. Sasaki \cite{sasaki_differentiable_1960} and then studied in the early 70s, Sasaki manifolds seem to have been more or less neglected until the early 90s. One can mention a paper by T. Friedrich and I. Kath published in 1990 \cite{friedrich_7-dimensional_1990} in which they gave a first classification result about Sasaki manifolds. From 1993 onwards, C. Boyer, K. Galicki and B. Mann have made important contributions to the understanding of the geometry and topology of Sasaki manifolds \cite{boyer_3-sasakian_1993, boyer_quaternionic_1993, boyer_3-sasakian_1998}. The year 1998 is a key milestone for Sasaki geometry: the influential paper by J. Maldacena \cite{maldacena1997large} who first proposed the AdS/CFT correspondance marked a significant regain of interest in Sasaki geometry. Indeed, manifolds that are product of anti-de Sitter space with Sasaki manifolds play a crucial role in AdS/CFT correspondance \cite{acharya1998branes, morrison1998nonspherical, Martelli_2008}. Finding examples, obstructions or sufficient conditions for the existence of Sasaki-Einstein metrics (i.e. Sasaki metrics for which the Ricci tensor is proportional to the metric) has led to a large exploration of Sasaki geometry \cite{boyer1998sasakianeinstein, boyer_sasakian_2007, futaki_transverse_2009, Martelli_2008, gauntlett2004, boyerEinsteinSphere2005}. As Einstein metrics are very particular versions of constant scalar curvature metrics or even extremal metrics \cite{CalabiExtremal1982, boyer_extremal_2013} it seems natural to study these more general metrics in the Sasaki world.

Recall that a Sasaki manifold $(M,g)$ is an odd dimensional Riemannian manifold whose metric cone $C(M) = (\R^*_+ \times M, dr^2+r^2 g)$ is Kähler. This synthetic description hides the extremely rich structure of Sasaki manifolds. In particular, $M$, which can be identified with the link $\left\lbrace r = 1 \right\rbrace \subset C(M)$, is a contact manifold with contact form $\eta = 2d^c \log(r)$. It defines a contact bundle $\ker \eta$ on which $\frac{1}{2}d\eta$ is a transverse Kähler form.  Here, $d = \partial + \overline{\partial}$ is the usual decomposition of the differential operator on a Kähler manifold and $d^c$ is defined as $d^c = \frac{i}{2}(\overline{\partial} - \partial)$. Any Sasaki manifold is endowed with a special vector field: the Reeb vector field $\xi$ which is the restriction of $J(r\partial_r)$ to the link $\left\lbrace r = 1 \right\rbrace$. Here $J$ denotes the complex structure on the Kähler metric cone $C(M)$. The restriction $\Phi$ of $J$ to the transverse distribution $\ker\eta$ is called a transverse complex structure. We call $(\xi, \eta, \Phi)$ a \emph{Sasaki structure}.

As a Sasaki manifold is trapped between its Kähler metric cone and its Kähler transverse structure, one can expect that these special metrics we are looking for are closely related to their Kähler counterparts. Kähler-Einstein metrics (i.e. Kähler metrics with Ricci form proportional to the metric itself) have been at the core of intense research over the past forty years \cite{aubin_equation_1978, chen2012kahlereinstein1, chen2012kahlereinstein2, chen2013kahlereinstein, Futaki1983, Tian1997, Tian2015, yau78}. This problem boils down to a non-linear second order PDE: a Monge-Ampère equation \cite{guedj_degenerate_2017}. Kähler-Einstein metrics are examples of constant scalar curvature metrics. The constant scalar curvature Kähler (cscK) metric problem of looking for cscK metrics was initiated by E. Calabi \cite{Calabi1985} and it boils down to a fourth order equation \cite{abreu1997kahler}, it also led to several works (mention for example \cite{berman2014convexity, DonaldsonScalar2001}) until recent major breakthrough by X. Chen and J. Cheng \cite{chen_constant_2017,  chen_constant_2018, chen_constant_2018-1}. The study of cscK metrics requires a deep understanding of the geometry of the space of Kähler metrics in a given Kähler class on a Kähler manifold $(X,\omega)$ \cite{donaldson_Kahler_1999, semmes_complex_1992}, identified with:
\begin{equation}
\label{eq:SpaceOfKahlerPotentials}
\mathcal{H}(X,\omega) = \left\lbrace \phi\in\mathcal{C}^{\infty}(X) \, \vert \, \omega_\phi := \omega + dd^c\phi > 0 \right\rbrace.
\end{equation}
Given the Mabuchi metric \cite{mabuchi_symplectic_1986} on the tangent space at a given $\phi\in\mathcal{H}(X,\omega)$ as being:
\begin{equation}
\label{eq:MabuchiMetric}
\scal{\psi_1}{\psi_2}_\phi = \int_M (\psi_1 \psi_2) \omega_\phi^n \quad \text{for } \psi_1, \psi_2\in T_{\phi}\mathcal{H}(X,\omega) \simeq \mathcal{C}^\infty(X),
\end{equation}
one can consider geodesics between two elements of $\mathcal{H}(X,\omega)$. X. Chen and his collaborators worked intensively in this direction \cite{chen_space_2000, calabi_space_2002, Chen_Tian_2008, Chen_Space_2009, Chen_Sun_2009} proving in particular that this infinite dimensional space is a path metric space with $\mathcal{C}^{1,\overline{1}}$ geodesics. Note that Chu-Tosatti-Weinkove \cite{Chu_regularityC11_2018} has established that geodesics are $\mathcal{C}^{1,1}$-regular, which is known to be optimal regularity thanks to examples of T. Darvas, L. Lempert, L. Vivas \cite{darvas2014morse, darvas_weak_2012, lempert2013geodesics}. T. Darvas then consequently refined the study of the geometry of the space of Kähler metrics \cite{darvas_mabuchi_2014, darvas2017weak, darvas2019isometries} especially identifying its metric completion with a space of weighted finite energy class $\mathcal{E}^2(X, \omega)$ (previously introduced in \cite{guedj_weighted_2007}) and showing that it is non-positively curved in the sense of Alexandrov. For further references and details about $\mathcal{E}^2(X, \omega)$, we refer to \cite{guedj_degenerate_2017}.

Theses advances in the Kähler setting were truly inspirational for the Sasaki setting. In \cite{Martelli_2006, Martelli_2008, futaki_transverse_2009, boyer_sasakian_2007, collins2019sasaki, li2021notes}, Sasaki-Einstein metrics are studied while constant scalar curvature Sasaki metrics are studied in \cite{Legendre_2011, collins2012ksemistability, van_coevering_monge-amp`ere_2015, guan_regularity_2009, apostolov2021weighted, boyer2018relative, he2014transverse}. In this direction, people considered the space of potentials:
$$
\mathcal{H}(M, \xi, d\eta) = \left\lbrace \phi\in\mathcal{C}^{\infty}_B(M),\; d\eta + dd^c\phi > 0 \right\rbrace,
$$
where $\mathcal{C}^{\infty}_B(M)$ is the space of smooth \emph{basic} functions (ie smooth functions which are invariant under the Reeb flow). Analogously to the Kähler setting, the $d^c$-operator, acting on basic functions, is defined as $\frac{i}{2}(\overline{\partial}_B - \partial_B)$ where the $\partial_B$ and $\overline{\partial}_B$ operator are defined in \cite[Section 7]{boyer_sasakian_2007}, see also Section \ref{section:BasicForms}. We omit the subscript $B$ for simplicity. As we will explain in Section \ref{section:BasicForms}, any potential in $\mathcal{H}(M, \xi, d\eta)$ defines a new Sasaki structure on $M$. This infinite dimensional space, whose tangent space at any $\phi\in\mathcal{H}(M, \xi, d\eta)$ is identified with $\mathcal{C}_B^\infty(M)$ is endowed with a Riemannian structure, analogue of the Mabuchi metric \cite{guan_geodesic_nodate, guan_regularity_2009}:
$$
\scal{\psi_1}{\psi_2}_\phi = \int_M (\psi_1 \psi_2) \; \eta\wedge (d\eta + dd^c\phi)^n.
$$
P. Guan and X. Zhang \cite{guan_geodesic_nodate, guan_regularity_2009} proved the existence of $\mathcal{C}^{1,\overline{1}}$ geodesics (Proposition \ref{S.prop:epsGeoJacobi}) using a Monge-Ampère type re-formulation for the geodesic equation (see Sections \ref{Section:GeodesicEquation} and \ref{Section:WeakGeodesics}). They also showed that the Riemannian structure on the tangent space of $\mathcal{H}(M, \xi, d\eta)$ induces a metric $d$ on $\mathcal{H}(M, \xi, d\eta)$:
$$
d(\phi_0, \phi_1) := \inf\left\lbrace   \int_0^1 \sqrt{\scal{\dot{\phi_t}}{\dot{\phi_t}}_{\phi_t}}dt \; ; \; t\mapsto\phi_t \text{ is a smooth path joining } \phi_0 \text{ to } \phi_1 \right\rbrace.
$$
This definition of $d$ is natural but showing that this is indeed a distance is not as easy as for finite dimensional manifolds. W. He and J. Li generalised most of the geometrical results known in the Kähler case to Sasaki setting \cite{he_geometrical_2018} allowing W. He to extend X. Chen and J. Cheng result for constant scalar curvature Sasaki metrics \cite{he2018scalar}. W. He and J. Li \cite{he2018scalar} then used pluripotential theory to study the metric completion of $(\mathcal{H}(M, \xi, d\eta), d)$ and its geometry. Using C. Van Coevering work \cite{van_coevering_monge-amp`ere_2015}, they basically generalized the results known in the Kähler setting \cite{guedj_degenerate_2017}. In their study of the geometry of $\mathcal{H}(M,\xi,d\eta)$, \emph{energy classes} will play a crucial role. The first energy class to be considered is $\mathcal{E}(M, \xi, d\eta)$. This is the space of all quasi-plurisubharmonic functions with full Monge-Ampère mass $\left(\text{i.e. }\int_M \eta\wedge (d\eta + dd^c\phi)^n = \int_M \eta\wedge d\eta^n\right)$. Building on this one can consider the energy class:
$$
\mathcal{E}^{2}(M,\xi,d\eta) := \left\lbrace \phi \in \mathcal{E}(M, \xi, d\eta) \; ; \; \int_M (\phi^2) \eta\wedge(d\eta + dd^c\phi)^n < \infty \right\rbrace.
$$
We refer to \cite{he_geometrical_2018, van_coevering_monge-amp`ere_2015}, to Section \ref{sec:psh} and to Section \ref{Section:DistanceHAndExtention} for the notions of quasi-plurisubharmonicity and these energy classes. Our main result states as follows:
\begin{ThA}
\label{Th:MainTHIntro}
The metric completion, $\mathcal{E}^2(M,\xi,d\eta)$, of $\left(\mathcal{H}(M, \xi, d\eta), d\right)$ is negatively curved in the sense of Alexandrov.
\end{ThA}

The statement above is the analogue of what was already known in the Kähler case. Its proof does not contain any original idea. Nevertheless, there was a technical work that had to be done in order to ensure that the usual analogy between the Kähler and the Sasaki setting still holds for the CAT(0) property. Note that not all the results established in the Kähler case can be straightforwardly adapted to the Sasaki case. For example, defining a notion of $K$-stability on Sasaki manifolds involved works by Ross-Thomas \cite{ross2007study} and Collins-Székelyhidi \cite{collins2018k}.

We organise this note as follows: we first recall, in Section \ref{Section:SasakianGeometry}, some definitions about Sasaki manifolds in order to fix notations. Then we define the set of Sasaki potentials $\mathcal{H}(M, \xi, d\eta)$ and give a geometrical interpretation of this space (Proposition \ref{Prop:HdescribesSasakiStructures}). Section \ref{Section:GeometrySmooth} is devoted to introduce the analogue of the Mabuchi metric on $\mathcal{H}(M, \xi, d\eta)$. In Section \ref{Section:GeodesicEquation} we give equivalent formulations for the geodesic equation allowing to weaken the notion of geodesics. Finally, in Section \ref{Section:GeometryMetricCompletion} we prove Theorem \ref{Th:MainTHIntro}.

\section*{Acknowledgements}
\footnotesize
\addcontentsline{toc}{section}{Acknowledgements}
Je tiens à remercier tout particulièrement mes deux encadrants, Eleonora et Gilles. Ils ont su me guider et m'accompagner sans relâche tout au long de ce travail. Leurs conseils avisés m'ont été d'une aide incroyable dans ma compréhension. Merci aussi pour le tout le temps que vous m'avez consacré, pour tout ce que j'apprends avec vous et pour tous les bons moments passés à vos côtés. Mes remerciements également à L. Lempert pour avoir pris le temps de m'écouter et pour ses questions. Elles m'éclairent et me donnent des pistes très pertinentes pour la suite.

Je tiens aussi évidemment à remercier ma famille et mes proches qui sont d'un soutien inébranlable depuis toujours.

I thank the anonymous referee for their valuable advice and the rewarding discussion we had.
\normalsize

\section{Sasaki Geometry and Smooth Potentials}
\label{Section:SasakianGeometry}

This section starts with some preliminaries in Sasaki Geometry: we fix notations and then define the space of Sasaki potentials. We refer to \cite{boyer_sasakian_2007} for an extensive study of Sasaki manifolds.

\subsection{Sasaki Manifolds}
\label{Sec:SasakianManifolds}
We consider $(M,\xi, \eta, \Phi, g)$ a compact real smooth manifold of dimension $2n+1$, where $(M,\xi,\eta)$ is a contact manifold (i.e. $\eta$ is a contact form and $\xi$ the Reeb vector field: $\eta(\xi) = 1$ and $\iota_\xi d\eta = 0$), $g$ is a Riemannian metric and $\Phi$ a $(1,1)$-tensor field with the following compatibility conditions:
$$
\Phi\circ\Phi = -\mathds{1}_{TM} + \xi \otimes \eta \; ; \qquad g\circ(\Phi\otimes \mathds{1}_{TM}) = \frac{1}{2}d\eta \; ; \qquad g\circ(\Phi\otimes\Phi) = g - \eta\otimes\eta.
$$
Note that $\eta\circ\Phi = 0$ ; $\Phi(\xi) =0$ and $g$ is completely determined by $\eta$ and $\Phi$:
\begin{equation}
\label{S.eq:MetriqueRiemannienneEnFonctionDuReste}
g = \eta\otimes\eta +\frac{1}{2}d\eta\circ(\mathds{1}_{TM}\otimes\Phi).
\end{equation}

A \emph{Sasaki Manifold} is such a manifold with an additional integrability condition. The purpose of the next section is to formulate this condition on the symplectization of $M$.

\subsubsection{Metric Cone}
\label{sec:MetricConeSasakianManifolds}
Given such a manifold $(M,\xi, \eta,\Phi, g)$, one can construct a metric cone $C(M, \eta)$ (called symplectization), also denoted $C(M)$ if there is no ambiguity (see \cite[Appendix 4 - E]{arnoldMathematical}):
\begin{align*}
C(M,\eta) &:= \left\lbrace \alpha\in T^{\star}_x M , x\in M \vert \ker\alpha = \ker\eta_x \, , \; \alpha \text{ and } \eta_x \text{ defining the same orientation} \right\rbrace
\end{align*}
This set is furnished with a symplectic structure which is basically the restriction of $d\tau$ to $C(M,\eta)\subset T^\star M$ where $\tau$ is the canonical one-form on the cotangent bundle. We have a canonical identification of $C(M,\eta)$ with $M\times \R_+^{\star}$:
\begin{align*}
C(M,\eta) &\longrightarrow M\times \R_+^{\star} \\
\alpha \in T_x^\star M &\longmapsto (x, \sqrt{\alpha(\xi_x)}) =: (x,r).
\end{align*}
In $C(M,\eta)$, one has the so called \emph{canonical identification} $M\simeq \left\lbrace r = 1 \right\rbrace \subset C(M)$. We have a projection map:
$$
\pi_r : C(M,\eta) \rightarrow \left\lbrace r = 1 \right\rbrace.
$$
From now on, we consider $M$ as being $\left\lbrace r = 1 \right\rbrace$, and assume that $M$ is furnished with $(\xi, \eta, \Phi, g)$ as in Section \ref{Sec:SasakianManifolds}. Let $g_C := dr^2 + r^2 (\pi_r^{\star}g)$ be a metric on $C(M,\eta)$. For this metric, we let $\psi$ be the gradient of $\frac{r^2}{2}$ and we extend the Reeb vector field: $\overline{\xi} = (\xi, 0)$. Using these two vector fields and the canonical identification, we define an almost complex structure on $C(M,\eta) \simeq M\times\R_+^{\star}$:
$$
\begin{cases}
I\psi = \overline{\xi}\\
I(Y,0) = (\Phi(Y),0) -\eta(Y)\psi, \text{ where } Y \text{ is a tangent vector to } M.
\end{cases}
$$
If the almost complex structure $I$ on $C(M)$ is integrable, then we call $(\xi, \eta, \Phi, g)$ a \emph{Sasaki} structure. We say that $M$ is a \emph{Sasaki manifold} if it can be endowed with a Sasaki structure. In particular, given a Sasaki manifold, the almost complex structure defined above is upgraded to a Kähler structure. The next proposition outlines this Kähler structure:

\begin{Prop}[{\cite[Section 6.5]{boyer_sasakian_2007}}]
\label{S.Prop:SasakianCone}
Let $M$ be a Sasaki manifold. Set $\overline{\eta} := \pi_r^\star \eta$ and $\omega_C := \frac{1}{2}d(r^2 \overline{\eta})$. Then, $(C(M), g_C, \omega_C, I)$ is a K\"ahler manifold. Moreover, $\overline{\eta} = 2d^c\log(r) = \frac{2}{r}d^c r$ and $\omega_C = dd^c\left(\frac{r^2}{2}\right)$.
\end{Prop}

\subsubsection{Kähler Cone}
The complex structure, defined in Section \ref{sec:MetricConeSasakianManifolds}, on the symplectization of a Sasaki manifold is actually a Kähler structure. Here, we first define what we call a \emph{Kähler cone metric} and then state a correspondence between these special Kähler metrics and Sasaki structures.

\begin{Def}
\label{DefKahlerConeMetric}
Given a complex manifold $(C, I)$, a \emph{Kähler cone metric} on $(C,I)$ is a $(1,1)$-form of the form $dd^c\left(\frac{r^2}{2}\right)$ where $r:C\rightarrow \R_+^{\star}$ is a positive function such that $\left\lbrace r = 1 \right\rbrace$ is compact and such that:
\begin{enumerate}
\item $dd^c\left(\frac{r^2}{2}\right)$ is Kähler,
\item The radial vector field $\psi := \nabla(\frac{r^2}{2})$ is holomorphic with respect to $I$ (i.e. $\Lie_\psi I = 0$),
\item\label{DefKahlerConeMetric3} $g_C(\psi,\psi) = r^2$.
\end{enumerate}
Here, $g_C$ stands for the Riemannian metric associated to $dd^c\left(\frac{r^2}{2}\right)$ and $\nabla$ stands for the gradient according to $g_C$. We say that such a $C$ is a \emph{Kähler cone}.
\end{Def}

Proposition \ref{S.Prop:SasakianCone} says that given a Sasaki manifold, we have a Kähler cone metric on its symplectization $C(M,\eta)$. On the other hand, a Kähler cone metric induces a Sasaki structure on $M = \left\lbrace r = 1 \right\rbrace$. Indeed, the flow of $\psi$ gives a projection $\pi : C \rightarrow \left\lbrace r = 1 \right\rbrace$ and a decomposition of $C$ as a Riemannian cone in the sense of \cite[Definition 6.5.1]{boyer_sasakian_2007}: $C \simeq \left\lbrace r = 1 \right\rbrace \times \R_+^\star$ with the metric $dr^2 + r^2 \pi^\star (g_{\vert M})$. We set $\xi := \pi_\star(I\psi)$, $\eta = 2d^c\log(r)_{\vert \left\lbrace r = 1 \right\rbrace}$ and define $\Phi$ as being the restriction of $I$ to $\ker \eta$ and $\Phi(\xi) = 0$. It is straightforward to check that $(M, \xi, \eta, \Phi, g_{\vert M})$ is a Sasaki manifold. We summarize this discussion:

\begin{Prop}
\label{Prop:OneToOneKahlerConeAndSasaki}
There is a one-to-one correspondence between compact Sasaki manifolds and Kähler cones.
\end{Prop}

\subsection{Basic Forms and Potentials}
\label{section:BasicForms}
Here, we recall the definition of basic functions on Sasaki manifolds in order to provide a nice description of Kähler cone metrics in terms of basic functions (Proposition \ref{Prop:HdescribesSasakiStructures}).
In the sequel, $M$ is a compact Sasaki manifold and we use the notations introduced in Section \ref{Sec:SasakianManifolds}.

\begin{Def}
We say that a $p$-form $\alpha$ on $M$ is \emph{basic} if $\iota_\xi \alpha = 0$ and $\Lie_\xi \alpha = 0$. 
\end{Def}

In the case of $0$-forms we set:
$
\mathcal{C}^{\infty}_B(M) := \left\lbrace \phi \in \mathcal{C}^{\infty}(M), \Lie_{\xi} \phi =0 \right\rbrace.
$

Following \cite[Section 7]{boyer_sasakian_2007}, one can define basic operators $d_B, \partial_B, \overline{\partial}_B$ and their associated cohomologies. In this context, a $\partial_B \overline{\partial}_B$-lemma holds:
\begin{Lemma}[{\cite[Lemma 7.5.6]{boyer_sasakian_2007}}]
\label{lemma:partialPartialBarBasic}
Let $M$ be a compact Sasaki manifold. Let $\omega$ and $\omega'$ be closed, basic cohomologous $(1,1)$-forms. Then there exists a function $\phi \in \mathcal{C}^{\infty}_B(M)$ such that
$
\omega = \omega' +i\partial_B\overline{\partial_B}\phi.
$
\end{Lemma}
We refer to \cite{boyer_sasakian_2007} for a proof and for many other properties of these operators. As in the K\"ahler case we define the $d_B^c$ operator (a real operator):
$
d_B^c = \frac{i}{2}(\overline{\partial_B} - \partial_B)
$
so that $d_Bd_B^c = i\partial_B\overline{\partial_B}$. Basic operators coincide with the usual ones on basic forms so we will often omit the subscript $B$.

We say that two Sasaki structures with the same Reeb vector field $(\xi, \eta, \Phi, g)$ and $(\xi, \eta', \Phi', g')$ on $M$ have \emph{the same transverse structure} if the following diagram commutes \cite[Section 7.5.1]{boyer_sasakian_2007}.
$$
  \xymatrix{
    TM \ar[r]^-p \ar[d]_\Phi & TM/L_{\xi} \ar[d]^J & TM \ar[l]_-p \ar[d]^{\Phi'}\\
    TM \ar[r]_-{p} 	        & TM/L_{\xi}          & \ar[l]^-p  TM
  }
$$
Here, $p$ is the natural projection, $J$ is the map induced by $\Phi$ (defined by the right hand side of the diagram) and $L_{\xi}$ is the line bundle generated by $\xi$.
Let's now compare two Sasaki structures $(\xi, \eta, \Phi, g)$ and $(\xi', \eta', \Phi', g')$ on $M$ with the same Reeb vector field: $\xi = \xi'$ and having the same transverse structure. Note that this last condition is fundamental because we want to identify the basic $(1,1)$-forms in cohomology. Since $\eta$ and $\eta'$ have the same Reeb vector field, the 1-form $\eta - \eta'$ is basic. Thus $d\eta - d\eta'$ is an exact basic form. Lemma \ref{lemma:partialPartialBarBasic} gives a basic function $\phi$ such that
$
d(\eta' - \eta) = dd^c\phi.
$
Since $d\eta$ and $d\eta'$ are both real, $\phi$ is a smooth real function. This motivates the definition of the following set of the so called \emph{Sasaki potentials}:
$$
\mathcal{H}(M,\xi,d\eta) = \left\lbrace \phi \in \mathcal{C}^{\infty}_B(M), d\eta_\phi = d(\eta + d^c\phi) > 0 \right\rbrace.
$$
In the sequel, when there is no ambiguity, we will write $\mathcal{H}$ for the space of Sasaki potentials $\mathcal{H}(M,\xi,d\eta)$.

\begin{Example}
\label{Ex:PotentialsOnSphere}
Consider the standard Sasaki structure on $\mathbb{S}^3$. Let $\overline{\phi}\in\mathcal{H}$ be a smooth Sasaki potential. Since it is basic, one can find $\phi\in\mathcal{C}^{\infty}(\mathbb{CP}^{1})$ such that $\phi\circ H = \overline{\phi}$, where $H$ stands for the Hopf fibration. Indeed, the orbits of the Reeb vector field are given by the Hopf fibration. Since $d\eta_{\overline{\phi}} > 0$, one has:
$
H^{\star}(2\omega_{FS} + dd^c\phi) = d\eta + dd^c(H^{\star}\phi) = d\eta_{\overline{\phi}} > 0.
$
Thus, $\phi\in\mathcal{H}(\mathbb{CP}^1,2\omega_{FS})$ (see \eqref{eq:SpaceOfKahlerPotentials}). 
\end{Example}

\begin{Remark}
\label{S.Rk:VolumeFormSimplification}
Using the fact that the $(2n+1)$'th basic cohomology group is trivial on a $(2n+1)$-dimensional Sasaki manifold \cite[Proposition 7.2.3]{boyer_sasakian_2007}, one gets:
$$
\eta_\phi \wedge (d\eta_\phi)^n = \eta\wedge(d\eta_\phi)^n,
$$
since $d^c_B \phi$ is basic and so is $d\eta_\phi$.
\end{Remark}

We note that
$$d\eta_\phi > 0 \iff \eta_\phi \wedge d\eta_\phi ^n \neq 0.$$
Indeed, take a minimizing point $p$ for $\phi$. At $p$, since $d\eta > 0$, we have $d\eta_{\phi}\vert_p > 0$.
If $\eta_\phi \wedge d\eta_\phi ^n \neq 0$, then by continuity $d\eta_\phi > 0$ everywhere on $M$. On the other hand, if $d\eta_\phi > 0$, one can define a function $g$ such that: $\eta_\phi\wedge d\eta_\phi^n = \eta\wedge d\eta_\phi^n =: g(\eta\wedge d\eta^n)$. Suppose that $g(p) =0$ where $p\in M$. Then, on $\ker \eta_{\vert p}$ the $2$-form $d\eta_\phi$ is degenerate, indeed, $(d\eta_\phi^n)_{\vert p} = \iota_{\xi}(\eta\wedge d\eta_\phi^n)_{\vert p} = 0$. This is a contradiction with the positivity of the transverse Kähler form $d\eta_{\phi}$.

\begin{Prop}
\label{Prop:HdescribesSasakiStructures}
Let $(M, \xi, \eta, \Phi, g)$ be a compact Sasaki manifold. We fix the induced complex structure $I$ on the cone $C(M,\eta)$ so that $(C(M,\eta), I)$ is a complex manifold. We let $r$ be the function defined in Section \ref{sec:MetricConeSasakianManifolds} inducing a canonical identification of $M$ in $C(M,\eta)$. Then, there is a one-to-one correspondence between the space of Sasaki potentials, $\mathcal{H}$, and the set of Kähler cone metrics on the complex manifold $(C(M,\eta), I)$ with fixed radial vector field (i.e. if $r$ and $\tilde{r}$ are as in Definition \ref{DefKahlerConeMetric} then we ask $\psi = \tilde{\psi}$).
\end{Prop}

\begin{proof}
Given a function $u\in\mathcal{H}$, we set $\tilde{r} := r e^{\frac{u}{2}}$. It is straightforward to check that $\tilde{r}$ induces a Kähler cone metric. On the other hand, fixing the complex structure $I$, the condition $\psi = \tilde{\psi}$ implies that $u := 2\log\left(\frac{\tilde{r}}{r} \right)$ is a basic function on $M$. We now have to show that $u$ lies in $\mathcal{H}$. We let $G : \left\lbrace \tilde{r} = 1 \right\rbrace \rightarrow \left\lbrace r = 1 \right\rbrace$ be the restriction of the map $C(M,\eta) \ni (x,r) \mapsto (x,r e^{\frac{u}{2}}) \in C(M,\eta)$. The map $G$ allows us to identify two copies of $M$ in its cone. A direct computation shows $G_\star \tilde{\eta} = \eta + d^c u$, where $\tilde{\eta} = 2d^c \log(\tilde{r})$ is the contact form on $\left\lbrace \tilde{r} = 1 \right\rbrace$. This shows that $\eta + d^c u$ is a contact form coming from a Sasaki structure and thus $u\in\mathcal{H}$.
\end{proof}

\begin{Remark}
This has been proved by V. Apostolov, D. Calderbank and E. Legendre in \cite{apostolov2021weighted} in a more stylish way. Our Definition \ref{DefKahlerConeMetric} of Kähler cone metrics is equivalent to \cite[Definition 2.2]{apostolov2021weighted} and Proposition \ref{Prop:HdescribesSasakiStructures} above becomes equivalent to \cite[Lemma 2.8]{apostolov2021weighted}. In \cite{apostolov2021weighted}, the Sasaki manifold $M$ is seen as the quotient of the cone by the action of $\psi$ and they do the proper identifications between the copies of $M$ in its cone \cite[Lemma 2.4]{apostolov2021weighted} beforehand. In our pedestrian proof, we work on a preferred copy of $M$, namely $\left\lbrace r = 1 \right\rbrace$ and we properly identify it with $\left\lbrace \tilde{r} = 1 \right\rbrace$ meanwhile proving the one-to-one correspondence.
\end{Remark}

In particular, $u\in\mathcal{H}$ induces a new Sasaki structure $(\xi, \eta_u, \Phi_u, g_u)$ on $M$ with same Reeb vector field. It is completely determined by $u$ using the correspondences given in Proposition \ref{Prop:OneToOneKahlerConeAndSasaki} and Proposition \ref{Prop:HdescribesSasakiStructures}:
$$
\eta_u := \eta + d^c u \quad ; \quad \Phi_u := \Phi - \xi\otimes (d^c u \circ \Phi).
$$
The Riemannian metric $g_u$ is then determined by \eqref{S.eq:MetriqueRiemannienneEnFonctionDuReste}. The new Sasaki structure $(\xi,\eta_u, \Phi_u, g_u)$ has the same transverse complex structure.

Observe that given a compact Sasaki manifold $M$, Sasaki structures induced by functions in $\mathcal{H}$ have same volume (see \cite[Proposition 7.5.10]{boyer_sasakian_2007}):
$$
\int_M \eta\wedge d\eta^n = \int_M \eta_u \wedge d\eta_u^n.
$$
This plays an important role when normalizing the Monge-Ampère measure.

\section{The Geometry of Smooth Potentials}
\label{Section:GeometrySmooth}
In this section, following the work of P. Guan and X. Zhang \cite{guan_geodesic_nodate, guan_regularity_2009}, we present some results about the geometry of $\mathcal{H}$ and its geodesics.

Given $\phi\in\mathcal{H}$, we introduce a $L^2$-metric on the tangent space of $\mathcal{H}$ at $\phi$, for $\psi_1, \psi_2 \in T_{\phi}\mathcal{H} \simeq \mathcal{C}^{\infty}_{B}(M)$, we set:
\begin{equation}
\label{eq:DefMabuchiMetricAnalogue}
\scal{\psi_1}{\psi_2}_\phi = \int_M (\psi_1 \psi_2) \eta_\phi\wedge d\eta_\phi^n= \int_M (\psi_1 \psi_2) \eta\wedge d\eta_\phi^n.
\end{equation}

\begin{Example}
One can compute the metric on the space of Sasaki potentials for the sphere. For any two $\overline{f}, \overline{g}\in \mathcal{C}_B^{\infty}$, we note $f,g\in\mathcal{C}^{\infty}(\mathbb{CP}^1)$ such that $\overline{f} = f\circ H$ and $\overline{g} = g\circ H$ (see Example \ref{Ex:PotentialsOnSphere}). Denoting $\scal{\cdot}{\cdot}_{\mathbb{S}^3}$ the Riemannian metric on the space of Sasaki potentials of $\mathbb{S}^3$ and $\scal{\cdot}{\cdot}_{\mathbb{CP}^1}$ the one on the space of K\"ahler potentials on $(\mathbb{CP}^1, 2\omega_{FS})$ (see \eqref{eq:MabuchiMetric}) one has:
$$
\scal{\overline{f}}{\overline{g}}_{\mathbb{S}^3,\overline{\phi}}= 2\pi\scal{f}{g}_{\mathbb{CP}^1,\phi}.
$$
Indeed, the integrals of $\eta$ along each orbit of $\xi$ are equal to $2\pi$.
\end{Example}

Let $t\in[0,1]\mapsto\phi_t$ be a smooth path in $\mathcal{H}$ and $\psi_1, \psi_2 \in\mathcal{C}^\infty_B(M\times[0,1])$ tangent to $\phi$. Stokes' theorem gives (see also \cite[Proposition 1]{guan_geodesic_nodate}):
$$
\frac{d}{dt}\scal{\psi_1}{\psi_2}_{\phi} = \scal{\frac{d\psi_1}{dt} - \frac{1}{4}g_{\phi}(\nabla \psi_1, \nabla \dot{\phi})}{\psi_2}_{\phi} + \scal{\psi_1}{\frac{d\psi_2}{dt} - \frac{1}{4}g_{\phi}(\nabla \psi_2, \nabla \dot{\phi})}_{\phi},
$$
where $\nabla$ stands for the gradient associated to $g_\phi$.

\begin{Def}
\label{Def:Connection}
Let $\phi : t \in [0,1] \mapsto \phi(t)\in \mathcal{H}$ be a smooth path and $\psi$ tangent to $\phi$ identified with smooth basic functions on $M\times [0,1]$.
$$
\nabla_{\dot{\phi}}\psi := \dot{\psi} - \frac{1}{4}g_{\phi}(\nabla \psi, \nabla \dot{\phi}),
$$
where $\dot{\psi} = \frac{d\psi}{dt}$.
\end{Def}

We recall \cite[Proposition 2]{guan_geodesic_nodate} that the connection $\nabla$ is compatible with $\scal{.}{.}_{\phi}$ and torsion free. Additionally, and this is crucial for Theorem \ref{S.th:Cat(0)Smooth}, we have:

\begin{Prop}[{\cite[Theorem 1]{guan_geodesic_nodate}}]
The sectional curvature is non-positive.
\end{Prop}

\subsection{Geodesic Equations}
\label{Section:GeodesicEquation}
In this section we present different equivalent formulations for the geodesic equation in $\mathcal{H}$ and give an example on the $3$-sphere $\mathbb{S}^3$.
The natural geodesic equation in $\mathcal{H}$ is $\nabla_{\dot{\phi}}\dot{\phi} = 0$,  which writes:
\begin{equation}
\label{S.eq:GeodesicEquation}
\ddot{\phi} - \frac{1}{4}g_{\phi}(\nabla \dot{\phi}, \nabla \dot{\phi})=0.
\end{equation}
In \cite{godlinski_locally_2000}, it has been proved that at any point in $M$, one can choose a local system of coordinates $(\tau, z_1,...,z_n)\in (-\delta, \delta) \times V\subset \R \times \C^n$ such that (using Einstein summation convention):
\begin{equation}
\label{S.eq:StructureInCoord}
     \begin{cases}
        \xi = \frac{\partial}{\partial \tau} \\
        \eta = d\tau - i \left(\frac{\partial h}{\partial z_j}dz_j  - \frac{\partial h}{\partial \overline{z}_j}d\overline{z}_j\right)\\
        \Phi = i\left( X_j \otimes dz_j - \overline{X}_{j} \otimes d\overline{z}_j \right) \\
        g = \eta \otimes \eta + \frac{\partial^2 h}{\partial z_k \partial \overline{z}_j} (dz_k\otimes d\overline{z}_j + d\overline{z}_j\otimes dz_k),
     \end{cases}
\end{equation}
where $h$ is a local, real valued, basic function (i.e. $\xi h =0$), such that $g$ is a positive definite and 
$$
X_j = \frac{\partial}{\partial z_j} + i \frac{\partial h}{\partial z_j} \frac{\partial}{\partial \tau} \; ;  \; 
\overline{X}_j = \frac{\partial}{\partial \overline{z}_j} - i \frac{\partial h}{\partial \overline{z}_j} \frac{\partial}{\partial \tau}.
$$
Now, for $\phi\in\mathcal{H}$, the induced Sasaki structure can be locally written in the same coordinate system as in \eqref{S.eq:StructureInCoord} replacing $h$ by $h_\phi := h + \frac{1}{2}\phi$. In such a coordinate system, the geodesic equation can be rewritten as:
\begin{equation}
\label{S.eq:GeodesicEquationCoord}
\ddot{\phi} - \frac{1}{2}h_{\phi}^{,j\overline{k}}\frac{\partial \dot{\phi}}{\partial z_k} \frac{\partial \dot{\phi}}{\partial \overline{z}_j} = 0.
\end{equation}
P. Guan and X. Zhang \cite[Proposition 2]{guan_regularity_2009} showed that the geodesic equation \eqref{S.eq:GeodesicEquation} can be reformulated as a Monge-Ampère type equation on the cone $C(M)$. Given $\phi_t$ be a smooth path in $\mathcal{H}$, we define $\psi$ on $M\times \left[1,\frac{3}{2}\right]$ by: 
\begin{equation}
\label{S.Def:PSI}
\psi(\cdot,r) := \phi_{2(r-1)}(\cdot) + 4\log(r).
\end{equation}
We set
$
\Omega_{\psi} := \omega_C + \frac{r^2}{2}\left(dd^c \psi - \frac{\partial \psi}{\partial r} dd^c r\right),
$
where $\omega_C$ is the K\"ahler form on the cone and $d$, $d^c$ are the usual operators on the cone.

\begin{Prop}[{\cite[Proposition 2]{guan_regularity_2009}}]
Fix $\varepsilon \geq 0$. The following Dirichlet problems are equivalent.
\begin{equation}
\label{S.eq:GeodesicsDirichletEpsilon}
	\begin{cases}
		\left(\ddot{\phi} - \frac{1}{4}g_{\phi}(\nabla \dot{\phi}, \nabla \dot{\phi})\right)\eta\wedge d\eta_{\phi}^n= \varepsilon \eta\wedge d\eta^n \text{ on } M\times(0,1)\\
		\phi\vert_{t=0} = \phi_0 \\
		\phi\vert_{t=1} = \phi_1.
	\end{cases}
\end{equation}

\begin{equation}
\label{S.eq:GeodesicDirichletMAEpsilon}
	\begin{cases}
		\Omega_{\Psi}^{n+1}= \varepsilon r^2 \omega_C^{n+1} \text{ on } M\times (1,\frac{3}{2}) \\
		\psi\vert_{t=1} = \psi_1 \\
		\psi\vert_{t=\frac{3}{2}} = \psi_{\frac{3}{2}}.
	\end{cases}
\end{equation}
\end{Prop}

\begin{Example}
Recall that, for a Kähler manifold $(X,\omega)$, the geodesic equation on $\mathcal{H}(X,\omega)$ is given, in a chart where $\omega = \frac{i}{2}\omega_{j\overline{k}}dz_j\wedge dz_{\overline{k}}$, by: $\ddot{\phi} - 2\omega_\phi^{j\overline{k}}\frac{\partial \dot{\phi}}{\partial z_j}\frac{\partial \dot{\phi}}{\partial \overline{z}_k} = 0$ (see \cite[Equation 2.4.1]{mabuchi_symplectic_1986} and \eqref{eq:SpaceOfKahlerPotentials} for notations). In the case of $\mathbb{CP}^1$, for both usual charts, the metric $\omega_{FS}$ is given by:
$$
\omega_{FS} = \frac{i}{2}\frac{dz\wedge d\overline{z}}{(1+z\overline{z})^2}.
$$
The computations are the same in both charts since here, we have: $\omega_{1\overline{1}} = (1+z\overline{z})^{-2}$. Recall that the Hopf fibration brings back $2\omega_{FS}$ to $d\eta$: $H^\star(2\omega_{FS}) = d\eta$. Thus, writing $d\eta$ in coordinates as in \eqref{S.eq:StructureInCoord} gives: $2h_{1\overline{1}} = \omega_{1\overline{1}}\circ H$. Therefore, $\frac{1}{2}h^{1\overline{1}} = \omega^{1\overline{1}}\circ H$. Pulling back the geodesic equation for $\phi\in\mathcal{H}(\mathbb{CP}^1,2\omega_{FS})$ by $H$ exactly gives the geodesic equation \eqref{S.eq:GeodesicEquationCoord} for $\phi\circ H$ in $\mathcal{H}$, the space of Sasaki potentials on $\mathbb{S}^3$. For this reason, finding a geodesic in $\mathcal{H}$ boils down to finding one in $\mathcal{H}(\mathbb{CP}^1,2\omega_{FS})$. In a chart of $\mathbb{CP}^1$, set:
$$
\phi_t := \log(1+e^{2t}\vert z\vert^2) - 2\log(1+\vert z\vert^2).
$$
This map is defined so that $2\omega_{FS} +dd^c\phi_t = dd^c\log(1+e^{2t}\vert z \vert^2)$. Now using the reformulation of \cite[Equation 2.4.1]{mabuchi_symplectic_1986} in terms of Monge-Ampère equation (see for example \cite[Section 15.2.2.1]{guedj_degenerate_2017}), we see that $\phi_t$ is a geodesic in $\mathcal{H}(\mathbb{CP}^1, 2\omega_{FS})$, indeed: on $\mathbb{C}^2$, $\left(dd^c \log(1+\vert\zeta\vert^2\vert z \vert^2)\right)^2 =0$. Thus $\phi_t\circ H$ is a geodesic in $\mathcal{H}$. Note that, on $\mathbb{CP}^1$, in terms of metrics, this geodesic goes from $\omega_{FS}$ to $C^\star \omega_{FS}$ where $C : [z_0:z_1]\mapsto [z_0:ez_1]$ is a conformal map on $\mathbb{CP}^1$.
\end{Example}

\section{Geometry of the Metric Completion of $\mathcal{H}$}
\label{Section:GeometryMetricCompletion}
\subsection{Plurisubharmonic Functions}
\label{sec:psh}

Here, we present the material we need about plurisubharmonicity. We refer to \cite{guedj_degenerate_2017} for an extensive reference about plurisubharmonicity.

\begin{Def}
\label{S.Def:PSH}
A function $u:M\rightarrow \R\cup\left\lbrace -\infty \right\rbrace$ is said to be (transverse) $d\eta$-plurisubharmonic ($d\eta$-psh) if $u$ is invariant under the Reeb flow, if $u$ is locally the sum of a smooth function and a plurisubharmonic function and: 
$$d\eta + dd^c u \geq 0,$$
in the sense of currents. We let $PSH(M, \xi, d\eta)$ be the set of all $d\eta$-plurisubharmonic functions which are not identically $-\infty$.
\end{Def}

The first result we state about this class of function is an approximation result analogous to the K\"ahler case. It will be used in the sequel.

\begin{Prop}[{\cite[Lemma 3.1]{he_geometrical_2018}}]
\label{S.Prop:ApproxPSH}
Given $u \in PSH(M, \xi, d\eta)$, there exists a sequence $u_k \in \mathcal{H}$ decreasing to $u$.
\end{Prop}

We can now define an analogue of the Monge-Amp\`ere measure for functions in $PSH(M,\xi,d\eta)$. For bounded plurisubharmonic function, C. Van Coevering \cite{van_coevering_monge-amp`ere_2015} adapted the Bedford-Taylor theory to the Sasaki setting, hence defining $\eta\wedge d\eta_u^n$ when $u\in PSH(M,\xi,d\eta)$ is bounded. As in the K\"ahler case, we extend the definition: for $u\in PSH(M, \xi,d\eta)$ we set $u_j := \max(u, -j)$. Following \cite[Definition 3.2]{he_geometrical_2018}, we set
$
\eta \wedge d\eta_u^n := \lim_{j\rightarrow \infty}\mathbf{1}_{\left\lbrace u > -j \right\rbrace} \eta\wedge d\eta_{u_j}^n.
$ Note that thanks to the maximum principle \cite[Proposition 3.2]{he_geometrical_2018}, this in an increasing sequence of measures. The limit is then a measure with total mass smaller that the total volume: $\int_M \eta\wedge d\eta^n$. We then define the set of functions with full Monge-Ampère mass:
$$
\mathcal{E}(M, \xi, d\eta) := \left\lbrace u \in PSH(M, \xi, d\eta) ; \int_M \eta \wedge d\eta_u^n = \int_M \eta\wedge d\eta^n \right\rbrace.
$$
At this point, we can define a special subset of $\mathcal{E}(M,\xi,d\eta)$. For any $u\in PSH(M, \xi, d\eta)$ we set $E(u) := \int_M \vert u \vert^2 \eta \wedge d\eta_u^n \in [0,+\infty].$
and we define the following finite energy class:
\begin{equation}
\label{eq:DefE2}
\mathcal{E}^2(M,\xi,d\eta) := \left\lbrace u\in\mathcal{E}(M,\xi, d\eta) ; E(u) < \infty \right\rbrace.
\end{equation}
We refer to \cite[Section 3]{he_geometrical_2018} for a deep study of finite energy class.

\subsection{Weak Geodesics}
\label{Section:WeakGeodesics}
In order to prove that the function $d$ on $\mathcal{H}\times\mathcal{H}$ defined in Section \ref{Section:DistanceHAndExtention} is a distance \cite[Theorem 3]{guan_regularity_2009}, P. Guan and X. Zhang proved, among others, a technical result \cite[Lemma 14]{guan_regularity_2009} in order to get the triangle inequality. This Lemma proves the existence of weak geodesics and gives an approximation with $\varepsilon$-geodesic. Following \cite{guan_regularity_2009} we define weak geodesics and $\varepsilon$-geodesics. Here and in the sequel, $\overline{M} := M\times\left[ 1, \frac{3}{2}\right] \subset C(M)$ and $\mathcal{C}^{1,\overline{1}}(\overline{M})$ is the closure of smooth function under the norm:
$ \norm{\cdot}_{\mathcal{C}^{1,\overline{1}}} := \norm{\cdot}_{\mathcal{C}^1(\overline{M})} + \sup_{\overline{M}}\vert \Delta \cdot \vert$, where $\Delta$ is the Riemannian Laplacian on $C(M)$.

\begin{Def}
For any $\phi_0, \phi_1 \in \mathcal{H}$, we say that :
\begin{enumerate}
\item $\phi_t$ is a \emph{weak geodesic} between $\phi_0$ and $\phi_1$ if the function $\psi =\phi_{2(r-1)} + 4\log(r)$ defined in \eqref{S.Def:PSI} is a weak solution to $\eqref{S.eq:GeodesicDirichletMAEpsilon}_{\varepsilon = 0}$ (i.e. $\psi$ is a bounded function such that $\Omega_\psi > 0$ and $\Omega_\psi^{n+1} =0$).
\item $\phi_t^{\varepsilon}$ is a $\varepsilon$-geodesic between $\phi_0$ and $\phi_1$ if $\psi^\varepsilon :=\phi^\varepsilon_{2(r-1)} + 4\log(r)$ satisfies $\Omega_{\psi^\varepsilon} > 0$ and \eqref{S.eq:GeodesicDirichletMAEpsilon}.
\end{enumerate} 
\end{Def}

\begin{Prop}[{\cite[Theorem 1]{guan_regularity_2009}}]
For any smooth $\phi_0, \phi_1 \in \mathcal{H}$, there exists a unique $\mathcal{C}^{1,\overline{1}}(\overline{M})$ weak geodesic between $\phi_0$ and $\phi_1$.
\end{Prop}

This result has been extended to \cite[Lemma 14]{guan_regularity_2009} which will be crucial in the sequel.

\begin{Prop}[{\cite[Lemma 14]{guan_regularity_2009}}]
\label{S.prop:epsGeoJacobi}
Let $\varphi_0, \varphi_1$ be smooth paths in $\mathcal{H}$: $\varphi_i : s\in [0,1] \mapsto \varphi_i(\cdot, s)\in \mathcal{H}$ ($i=1,2$). For $\varepsilon_0$ small enough, there exist a unique two parameter smooth families of curves 
\begin{align*}
\varphi : [0,1] \times [0, 1] \times (0, \varepsilon_0] &\longrightarrow \mathcal{H}\\
			(t,s,\varepsilon) &\longmapsto  \varphi(\cdot,t,s,\varepsilon)
\end{align*}
 such that the following hold:
\begin{enumerate}[label = \roman*.]
\item Setting $\psi_{s,\varepsilon}(r,\cdot) := \varphi(\cdot, 2(r-1), s, \varepsilon) + 4\log(r)$, $\psi$ verifies \eqref{S.eq:GeodesicDirichletMAEpsilon} and $\Omega_\psi > 0$. In particular, for fixed $s$, we get a \emph{$\varepsilon$-geodesic} between $\varphi_0(\cdot,s)$ and $\varphi_1(\cdot,s)$.
\item There exists a uniform constant $C$ which depends only on $\varphi_0$ and $\varphi_1$ such that:
$$
\vert \varphi \vert + \left\vert \frac{\partial \varphi}{\partial s} \right\vert + \left\vert \frac{\partial \varphi}{\partial t} \right\vert < C\quad ; \quad 0 < \frac{\partial^2 \varphi}{\partial t^2} < C\quad ; \quad \frac{\partial^2 \varphi}{\partial s^2} < C.
$$
\item For fixed $s$, the $\varepsilon$-approximating geodesic $\varphi(\cdot,t,s,\varepsilon)$ converges, when $\varepsilon\rightarrow 0$, to the unique geodesic between $\varphi_0(\cdot,s)$ and $\varphi_1(\cdot,s)$ in $\mathcal{C}^{1,\overline{1}}$-topology.
\end{enumerate}
\end{Prop}

\begin{Prop}[{\cite[Section 2.4]{van_coevering_monge-amp`ere_2015}}]
\label{S.prop:EpsilonGeodesicDecreases}
The $\varepsilon$-geodesics are decreasing to the weak geodesic.
\end{Prop}

\subsection{Distance on $\mathcal{H}$ and on $\mathcal{E}^2$}
\label{Section:DistanceHAndExtention}
In this section, following \cite{guan_regularity_2009, he_geometrical_2018}, we define a natural distance $d$ on $\mathcal{H}$ for which we give a nice expression in Proposition \ref{S.Prop:DistanceLimiteEpsilonGeodesics}. Then, we extend the distance $d$ to $\mathcal{E}^2(M,\xi,d\eta)$.
\subsubsection{Distance for Smooth potentials}
Recall that we defined a Riemannian metric on $\mathcal{H}$ for which the length of a smooth path $\phi_t \in \mathcal{H}$ is given by:
$$
l(\phi) := \int_0^1 \sqrt{\int_M (\dot{\phi}_t)^2 \eta \wedge d\eta^n_{\phi_t}} dt.
$$
P. Guan and X. Zhang \cite[Theorem 3]{guan_regularity_2009} proved that the length of the weak geodesic induces a distance. A straightforward consequence of \cite[Theorem 3 and Equation (7.15)]{guan_regularity_2009} is that this length is equal to:
$$
d(\phi_0, \phi_1) = \inf\left\lbrace l(\phi) \; \vert \; \phi \text{ is a smooth path joining } \phi_0 \text{ and } \phi_1  \right\rbrace.
$$
In particular, $d$ is a distance on $\mathcal{H}$ and we have:

\begin{Prop}[{\cite[Equation (7.15)]{guan_regularity_2009}}]
\label{S.prop:distanceGeodesiquesEnergy1}
Let $\phi_0, \phi_1 \in \mathcal{H}$ and $\phi^{\varepsilon}_t$ be the $\varepsilon$-geodesic between $\phi_0$ and $\phi_1$ then:
$$
d(\phi_0, \phi_1) = \lim_{\varepsilon\rightarrow 0}\int_0^1 \sqrt{\int_M (\dot{\phi_t^{\varepsilon}})^2 \eta \wedge d\eta^n_{\phi^{\varepsilon}_t}} dt.
$$
\end{Prop}

The following will be of a great use in the proof of Theorem \ref{S.th:Cat(0)Smooth}.

\begin{Prop}
\label{S.Prop:DistanceLimiteEpsilonGeodesics}
Given $\phi_0, \phi_1 \in \mathcal{H}$ and $\phi_t$ be the weak geodesic and $\phi^{\varepsilon}_t$ the $\varepsilon$-geodesic between $\phi_0$ and $\phi_1$, one has:
$$
\forall t \in [0,1], \; d(\phi_0, \phi_1)^2 = \int_M (\dot{\phi}_t)^2 \eta \wedge d\eta_{\phi_t}^n = \lim_{\varepsilon\rightarrow 0}\int_M (\dot{\phi}_t^{\varepsilon})^2 \eta \wedge d\eta_{\phi_t^{\varepsilon}}^n .
$$
\end{Prop}

\begin{proof}
P. Guan and X. Zhang \cite[Theorem 1]{guan_regularity_2009} showed that there exists a constant $C$ independent of $\varepsilon$ such that $\norm{\phi^{\varepsilon}_t}_{\mathcal{C}^2_w} \leq C$. We set
$
e^{\varepsilon}(t) := \int_M (\dot{\phi}_t^{\varepsilon})^2 \eta \wedge d\eta_{\phi_t^{\varepsilon}}^n .
$
Thus, using \eqref{S.eq:GeodesicsDirichletEpsilon} we have:
\begin{align*}
\frac{d e^{\varepsilon}}{dt} &= 2\scal{\nabla_{\dot{\phi}_t^{\varepsilon}}\dot{\phi}_t^{\varepsilon}}{\dot{\phi}_t^{\varepsilon}}_{\phi^{\varepsilon}_t} 
= 2\varepsilon \int_M \dot{\phi}_t^{\varepsilon} \eta\wedge d\eta^n
\leq 2\varepsilon C \Vol(M) := 2\varepsilon C \int_M \eta\wedge d\eta^n.
\end{align*}
For any fixed $t\in[0,1]$, this gives:
$
\vert l(\phi^{\varepsilon}) - \sqrt{e^{\varepsilon}(t)} \vert \rightarrow 0.
$
But, using the estimate of Proposition \ref{S.prop:epsGeoJacobi} and Ascoli theorem gives a subsequence such that $\dot{\phi}^{\varepsilon}_{t} \rightarrow {\phi}_t$ uniformly.

We also have the weak convergence of measures \cite[Proposition 3.1]{he_geometrical_2018}: Suppose that $u_j\in PSH(M, \xi, d\eta)\cap L^{\infty}$ decreases to $u\in PSH(M, \xi, d\eta)\cap L^{\infty}$ then 
$
\eta \wedge d\eta_{u_j}^n \rightarrow \eta \wedge d\eta_{u}^n
$
in the weak sense of measures.
This, with Proposition \ref{S.prop:distanceGeodesiquesEnergy1} gives the result. 
\end{proof}

\subsubsection{Extension of $d$ to $\mathcal{E}^2$}
As in the K\"ahler setting, using smooth approximations given by Proposition \ref{S.Prop:ApproxPSH} one can extend the distance defined on $\mathcal{H}$ to $\mathcal{E}^2(M,\xi,d\eta)$. Given $\phi_0, \phi_1 \in \mathcal{E}^2(M,\xi,d\eta)$, and $\phi_0^k, \phi_1^k \in \mathcal{H}$ decreasing respectively to $\phi_0, \phi_1$, we set
$
\tilde{d}(\phi_0, \phi_1) := \lim_{k\rightarrow\infty} d(\phi_0^k, \phi_1^k).
$
W. He and J. Li proved \cite[Lemma 4.6]{he_geometrical_2018} that the definition above is independent of the choice of the approximate sequences. They also extended T. Darvas results to Sasaki manifolds and in particular showed $(\mathcal{E}^2(M,\xi,d\eta), \tilde{d})$ is the metric completion of $(\mathcal{H}, d)$ (\cite[Theorem 2]{he_geometrical_2018}). Additionally, one can consider $t\mapsto \phi^k_t \in \mathcal{H}_{\Delta}$ the weak geodesic between $\phi^k_0$ and $\phi^k_1$. This is a decreasing sequence (this follows from the maximum principle \cite[Lemma 4.1]{he_geometrical_2018}). We set, for $t\in (0,1)$:
\begin{equation}
\label{S.def:WeakGeodesicInE2}
\phi_t := \lim_{k\rightarrow \infty} \phi^k_t.
\end{equation}

Using these notations, W. He and J. Li proved the following:

\begin{Prop}[{\cite[Lemma 4.7]{he_geometrical_2018}}]
\label{S.prop:geodesicsAreMetricGeodesics}
The map $t\mapsto \phi_t$ is a geodesic segment in the sense of metric spaces. In particular, for all $l\in (0,1)$,
$
\tilde{d}(\phi_0, \phi_l) = l \tilde{d}(\phi_0, \phi_1).
$
\end{Prop}

\begin{Prop}[{\cite[Lemma 4.11]{he_geometrical_2018}}]
\label{S.Prop:DistanceContinue}
Let $u\in \mathcal{E}^2(M,\xi,d\eta)$. If $u_k \in \mathcal{E}^2(M,\xi,d\eta)$ decreases to $u$ then,
$
d(u_k, u) \rightarrow 0. 
$
\end{Prop}

\subsection{Non-Positive Curvature}

We first prove a CAT(0)-type inequality for $\mathcal{H}$ (Theorem \ref{S.th:Cat(0)Smooth}) and then extend it to $\mathcal{E}^2(M, \xi, d\eta)$ (Theorem \ref{S.th:Cat(0)E}). The proof is adapted from the K\"ahler case done by E. Calabi and X. X. Chen \cite[Theorem 1.1]{calabi_space_2002} and T. Darvas \cite{darvas_mabuchi_2014}.

\begin{Def}[CAT(0) spaces {\cite{bridson_metric_1999}}]
\label{Def:CAT0}
A geodesic metric space $(X, d)$ is said to be \emph{non-positively curved} in the sense of Alexandrov if for any distinct points $q,r\in X$, there exists a geodesic $\gamma : [0,1] \rightarrow X$ joining $q, r$ such that for any $a\in \gamma$ and $p\in X$ the following inequality is satisfied:
\begin{equation}
\label{eq:Cat0inequality}
d(p,a)^2 \leq \lambda d(p,r)^2 + (1-\lambda) d(p,q)^2 - \lambda(1-\lambda) d(q,r)^2.
\end{equation}
Where $\lambda = \frac{d(q,a)}{d(q,r)} \leq 1$.
\end{Def}

In this manuscript, we only give this suitable definition of CAT(0) spaces. We refer to \cite{bridson_metric_1999} for more about CAT(0) spaces.

\begin{Th}
\label{S.th:Cat(0)Smooth}
Given $p,q,r\in \mathcal{H}$ and $\lambda \in (0,1)$. If we denote $\phi_{qr}$ the weak geodesic segment between $q$ and $r$ and $a\in\phi_{qr}$ such that $\lambda d(q,r) = \tilde{d}(q,a)$ then:
$$
\tilde{d}(p,a)^2 \leq \lambda d(p,r)^2 + (1-\lambda) d(p,q)^2 - \lambda(1-\lambda) d(q,r)^2.
$$
\end{Th}

It is worth mentioning that one cannot say that $\mathcal{H}$ is a CAT(0) space since the element $a$ above only lies in $\mathcal{H}_{\Delta}$ ($\mathcal{H}$ is not a geodesic metric space).

\begin{proof}
We fix $\varepsilon, \varepsilon' > 0$. And three potentials $p,q,r\in \mathcal{H}$.
Proposition \ref{S.prop:epsGeoJacobi} applied to the constant path $\phi_0 \equiv p$ and $\phi_1$ the $\varepsilon'$-geodesic joining $q$ to $r$ gives a two-parameter family of curves: $\phi(.,t,s,\varepsilon)$. Recall that for fixed $s$, the path $t\mapsto \phi(.,t,s,\varepsilon)$ is a $\varepsilon$ geodesic.
We denote $X=\frac{\partial \phi}{\partial t}$ and $Y = \frac{\partial \phi}{\partial s}$. Finally, we write $E(s)$ the total energy of the $\varepsilon$-geodesic between $p$ and $\phi(.,1,s,\varepsilon)$:
$
E(s) := \int_0^1 \scal{X}{X}_{\phi}dt.
$
Let's compute the first derivative of $E$.
\begin{align*}
\frac{1}{2}\frac{d E}{d s} &= \frac{1}{2}\int_0^1 \frac{\partial}{\partial s} \scal{X}{X}_{\phi} dt 
											= \int_0^1 \scal{\nabla_Y X}{X}_{\phi} dt
											= \int_0^1 \left( \frac{\partial}{\partial t}\scal{X}{Y}_{\phi} - \scal{\nabla_X X}{Y}_\phi \right)dt\\
											&= \scal{X}{Y}_{\phi}\vert_{t = 1} - \int_0^1 \scal{\nabla_X X}{Y}_\phi dt.
\end{align*}
The last term in the above equation can be written thanks to \eqref{S.eq:GeodesicsDirichletEpsilon} as:
$$
\int_0^1\int_M Y \nabla_X X \eta\wedge d\eta_\phi^n dt = \varepsilon\int_0^1\int_M \frac{\partial \phi}{\partial s} \eta\wedge d\eta^n dt.
$$
Thus,
\begin{align}
\frac{1}{2}\frac{d E}{ds} 
&= \scal{X}{Y}_{\phi}\vert_{t = 1} - \varepsilon\int_0^1\int_M \frac{\partial \phi}{\partial s} \eta\wedge d\eta^n dt.
\end{align}
Before computing the second derivative of $E$, we prove that, at $t=1$:
\begin{equation}
\label{S.eq:inegaliteSurYUsingNegativeCurvature}
\scal{Y}{\nabla_X Y}_{\phi} \geq \scal{Y}{Y}_{\phi}.
\end{equation}
Indeed, setting $H = \frac{\eta \wedge d\eta^n}{\eta\wedge d\eta_{\phi}^n}$, since the sectional curvature is negative and $\nabla_X Y = \nabla_Y X$, 
\begin{align*}
\frac{1}{2}\frac{\partial^2}{\partial t^2} \scal{Y}{Y}_{\phi} &= \scal{\nabla_Y X}{\nabla_X Y}_{\phi} + \scal{\nabla_X \nabla_Y X}{Y}_{\phi}\\
		&= \scal{\nabla_X Y}{\nabla_X Y}_{\phi} + \scal{\nabla_Y \nabla_X X}{Y}_{\phi} - K(X,Y) \\
		&\geq \scal{\nabla_X Y}{\nabla_X Y}_{\phi} + \varepsilon\scal{\nabla_Y \left(\frac{\eta \wedge d\eta^n}{\eta\wedge d\eta_{\phi}^n} \right)}{Y}_{\phi} \\
		&\geq \scal{\nabla_X Y}{\nabla_X Y}_{\phi} + \varepsilon \int_M \frac{\partial \phi}{\partial s}\left(\frac{\partial H}{\partial s} - \frac{1}{4} g_{\phi}\left(\nabla H, \nabla \frac{\partial \phi}{\partial s}\right) \right) \eta\wedge d\eta_{\phi}^n.
\end{align*}
On the other hand, derivating the identity $H\cdot (\eta \wedge d\eta_\phi^n) = \eta \wedge d\eta^n$ with respect to $s$ gives:
$$
\frac{\partial H}{\partial s} \eta\wedge d\eta_{\phi}^n + n H \eta\wedge dd^c \left(\frac{\partial \phi}{\partial s} \right)\wedge d\eta_{\phi}^{n-1} = 0.
$$
Thus the last term above simplifies in
\begin{align*}
-n\varepsilon&\int_M \frac{\partial \phi}{\partial s} H \eta \wedge dd^c\left(\frac{\partial \phi}{\partial s} \right)\wedge d\eta_{\phi}^{n-1} - \frac{\varepsilon}{4}\int_M \frac{\partial \phi}{\partial s} g_\phi \left(\nabla H, \nabla \frac{\partial \phi}{\partial s} \right) \eta\wedge d\eta_{\phi}^n \\
&= \frac{\varepsilon}{4} \int_M g_{\phi}\left(\nabla\left(\frac{\partial \phi}{\partial s} H\right), \nabla \frac{\partial \phi}{\partial s}\right) - \frac{\partial \phi}{\partial s} g_{\phi}\left(\nabla H, \nabla \frac{\partial \phi}{\partial s} \right)\eta\wedge d\eta_{\phi}^n\\
&= \frac{\varepsilon}{4}\int_M g_{\phi}\left(\nabla \frac{\partial \phi}{\partial s} ,\nabla \frac{\partial \phi}{\partial s}\right) \eta\wedge d\eta^n \geq 0.
\end{align*}
In the above computation we used Stokes' theorem. Therefore,
$$
\frac{1}{2}\frac{\partial^2}{\partial t^2} \scal{Y}{Y}_{\phi} \geq \scal{\nabla_X Y}{\nabla_X Y}_{\phi}.
$$
This shows that $t\mapsto\vert Y(t)\vert_{\phi}$ (norm associated to $g_{\phi}$) is convex. But $Y(0) = 0$ so we get the claim \eqref{S.eq:inegaliteSurYUsingNegativeCurvature}. We can now compute the second derivative of $E(s)$:
\begin{align*}
\frac{1}{2}\frac{d^2 E}{ds^2} &= \frac{d}{ds}\scal{X}{Y}_{\phi} - \varepsilon\int_0^1\int_M \frac{\partial^2 \phi}{\partial s^2}\eta\wedge d\eta^n dt \\
					&= \scal{\nabla_X Y}{Y}_{\phi}\vert_{t=1} + \scal{X}{\nabla_Y Y}_{\phi}\vert_{t=1} - \varepsilon\int_0^1\int_M \frac{\partial^2 \phi}{\partial s^2}\eta\wedge d\eta^n dt \\
					&\geq \scal{Y}{Y}_{\phi}\vert_{t=1} + \int_{M}\underbrace{\frac{\partial \phi}{\partial t}}_{\geq -C}\underbrace{\nabla_Y Y\eta\wedge d\eta_{\phi}^n}_{=\varepsilon' \eta\wedge d\eta^n}\vert_{t=1} - \varepsilon\int_0^1\int_M \underbrace{\frac{\partial^2 \phi}{\partial s^2}}_{\leq C}\eta\wedge d\eta^n dt \\
					&\geq \scal{Y}{Y}_{\phi}\vert_{t=1} - (\varepsilon+\varepsilon') C \Vol(M).
\end{align*}

Where we used the fact that $\phi_1 = \phi(., 1, s, \varepsilon)$ is an $\varepsilon'$-geodesic, and the estimates of Proposition \ref{S.prop:epsGeoJacobi}. But $\scal{Y}{Y}_{\phi}\vert_{t=1}$ is exactly the energy of the $\varepsilon'$-geodesic joining $q$ and $r$ which we denote $E^{\varepsilon'}_{(qr)}$. Thus, we finally have:
$$
\frac{1}{2}\frac{d^2 E}{ds^2} \geq E^{\varepsilon'}_{(qr)} - (v\varepsilon+\varepsilon') C \Vol(M).
$$
This implies:
$
E(s) \leq (1-s)E(0) + s E(1) + (s^2-s)\left( E^{\varepsilon'}_{(qr)} - (\varepsilon+\varepsilon') C \Vol(M) \right).
$
Now, we fix $s$ and let $\varepsilon \rightarrow 0$. Proposition \ref{S.Prop:DistanceLimiteEpsilonGeodesics} gives ;
$$
d(p, \phi_1(.,s))^2 \leq  (1-s)d(p,q)^2 + s d(p,r)^2 + (s^2-s)E^{\varepsilon'}_{(qr)} - \varepsilon' C\Vol(M) (s^2-s).
$$

But $\phi_1(.,s)$ being on the $\varepsilon'$-geodesic $\phi_1$, denoting $\phi_{qr}$ the weak geodesic segment between $q$ and $r$, we have that $\phi_1(.,s)$ increases to $\phi_{qr}(s)$ as soon as $\varepsilon'$ decreases to $0$. So Proposition \ref{S.Prop:DistanceContinue} gives that $d(p, \phi_1(.,s)) \rightarrow \tilde{d}(p, \phi_{qr}(s))$, and thus:
$$
\tilde{d}(p, \phi_{qr}(s))^2 \leq  (1-s)d(p,q)^2 + s d(p,r)^2 + (s^2-s)d(q,r)^2.
$$
The conclusion follows from Proposition \ref{S.prop:geodesicsAreMetricGeodesics}.
\end{proof}

\begin{Th}
\label{S.th:Cat(0)E}
$\mathcal{E}^2(M,\xi,d\eta)$ is a CAT(0) space.
\end{Th}

\begin{proof}
Let $p,q,r \in \mathcal{E}^2(M,\xi, d\eta)$. We consider decreasing approximations: $p^k, q^k, r^k \in \mathcal{H}$. Theorem \ref{S.th:Cat(0)Smooth} gives:
$$
\tilde{d}(p^k, \phi_{q^k r^k}(s))^2 \leq  (1-s)d(p^k,q^k)^2 + s d(p^k,r^k)^2 + (s^2-s)d(q^k,r^k)^2.
$$
where $\phi_{q^k r^k}\in \mathcal{H}_{\Delta}$ is the weak geodesic between $q^k$ and $r^k$ which decreases to the metric geodesic defined by (\ref{S.def:WeakGeodesicInE2}). Using Proposition \ref{S.Prop:DistanceContinue} gives the CAT(0) inequality for $\mathcal{E}^2(M, \xi, d\eta)$ since $\phi_{qr}$ is a geodesic in the metric sense according to Proposition \ref{S.prop:geodesicsAreMetricGeodesics}.
\end{proof}

\begin{footnotesize}
\addcontentsline{toc}{section}{Bibliography}
\bibliographystyle{abbrv}
\bibliography{Bibli}
\end{footnotesize}

\end{document}